\documentclass[12pt]{amsart}
\usepackage{amsmath, amsthm}
\usepackage[english]{babel}
\usepackage[T1]{fontenc}
\usepackage[latin1]{inputenc}
\usepackage{amssymb}
\usepackage{amscd}
\usepackage{latexsym}
\usepackage[bookmarksnumbered,plainpages,hypertex]{hyperref}

\newtheoremstyle{theorem}
  {12pt}          
  {12pt}  
  {\sl}  
  {\parindent}     
  {\bf}  
  {. }    
  { }    
  {}     
\theoremstyle{theorem}
\newtheorem{theorem}{Theorem}

\newtheorem{remark}[theorem]{Remark}
\newtheorem{proposition}[theorem]{Proposition}
\newtheorem{lemma}[theorem]{Lemma}

\linethickness{0.5pt}

\newcommand{\aG}{\alpha}
\newcommand{\sG}{\sigma}
\newcommand{\bG}{\beta}
\newcommand{\dG}{\delta}
\newcommand{\bN}{\mathbb{N}}

\title[On the distance between perfect numbers.]{On the distance between perfect numbers.}

\author{Ph. Ellia}
\address{Dipartimento di Matematica, 35 via Machiavelli, 44100 Ferrara}
\email{phe@unife.it}


\begin{document}
\maketitle

\begin{abstract}
If $n,m$ are distinct perfect numbers and if $|n-m|=\dG$ it is known that $\dG >1$. We show that $\dG \nmid n$. Then we prove that there exist infinitely many triangular numbers $\dG \equiv 3\pmod{12}$ such that $|x-y|=\dG$ has no solutions with both $x$ and $y$ perfect numbers.
\end{abstract}

\section{Introduction}	

An integer $n \in \bN$ is said to be perfect if $\sG (n)=2n$ where $\sG$ is the sum of divisors function. By results of Euler every even perfect number has the form $n = 2^{p-1}(2^p-1)$ where $2^p-1$ is prime, whereas every odd perfect number is of the form $n = q^{4b+1}.\prod p_i^{2a_i}$, $q,p_i$ distinct primes, $q \equiv 1 \pmod{4}$; in particular if $n$ is an odd perfect number, then $n \equiv 1 \pmod{4}$. It is still unknown if odd perfect numbers exist (for some recent results see for example \cite{N}, \cite{OR}).

In their very stimulating paper \cite{LP}, Luca and Pomerance show that if the $abc$-conjecture is true, then the equation $x-y=\dG$ has a finite number of solutions with $x,y$ perfect numbers when $\dG$ is odd or if $x$ and $y$ are both squarefull.

An obvious remark is in order here: if one could prove that the distance between two perfect numbers is always even, then it would follow that any perfect number is even.

It has also been proved that if $n$ and $m$ are perfect, then $|n-m| > 1$ (cf \cite{LC}).

Here we show that if $n$ and $n+\dG$, $\dG > 1$, are both perfect, then $\dG \nmid n$, see Proposition \ref{tNonDiv}. 

It is known that modular properties imply that an integer $\dG \equiv \pm 1\pmod{12}$ can't be the distance between two perfect numbers (Remark \ref{-+1mod12}). It is natural to ask if there are other infinite series of numbers $\dG$ such that the equation $|x-y| = \dG$ has no solutions with $x, y$ both perfect.

We show (Theorem \ref{Thm-infty}) that there exist infinitely many triangular numbers $\dG\equiv 3\pmod{12}$ with this property.

\section{A property of the distance between two perfect numbers.}

We first recall the following (\cite{LC}):

\begin{lemma}
\label{Dist>1}
If $n$ and $m$ are perfect numbers and if $\dG :=|n-m|$, then $\dG > 1$.
\end{lemma}

We have:

\begin{lemma}
\label{t even}
Let $n$ and $n+\dG$ be perfect numbers. If $\dG \geq 2$ is even, then $\dG \nmid n$.
\end{lemma}

\begin{proof} This is clear if $n$ is odd. So assume $n=2^{p-1}(2^p-1)$, $n+\dG = 2^{q-1}(2^q-1)$, where $P =2^p-1$ and $Q =2^q-1$ are primes and $q > p$. Then $n+\delta > 2n$, which implies that $\dG > n$, so $\dG \nmid n$.
\end{proof}

\begin{lemma}
\label{todd}
Let $n$, $n+\dG$, be perfect numbers. If $|\dG |>1$ is odd, then $\dG \nmid n$.
\end{lemma}

\begin{proof} 
We may assume that $n$ is even. So $n=2^{p-1}\cdot Q$, where $Q=2^p-1$ is prime. Then $|\dG |=Q$ and so $n+\dG = Q(2^{p-1}\pm 1)$. But by Euler, we must have $Q^2\mid n+\dG$, a contradiction.
\end{proof}

Gathering everything together, we get:

\begin{proposition}
\label{tNonDiv}
If $n,m$ are distinct perfect numbers and if $\dG :=|n-m|$, then $\dG \nmid n$.
\end{proposition}

\begin{remark}
If $n$, $m$ are perfect numbers, it is easy to show that $\dG :=|n-m|>2$.

Moreover if $\dG \equiv 3 \pmod{4}$ and if $n > 6$, it can be proved that $(\dG ,n)=1$.
\end{remark}

\section{Special values of the distance between two perfect numbers.}

As stated in the introduction, using known modular results we have:

\begin{remark}
\label{-+1mod12}
An integer $\dG \equiv \pm 1\pmod{12}$ can't be the distance between two perfect numbers.
\end{remark}

Indeed this follows from the fact that any even perfect number is $\equiv 6$ or $4\pmod{12}$ (Euler's theorem), while any odd perfect number is $\equiv 1$ or $9\pmod{12}$ by Touchard's theorem (\cite{T}, \cite{Hold}).
\vspace{0.3cm}

We introduce some notations:

In the sequel we denote by $\dG = b(b-1)/2$ a triangular number such that:

$b = 4m+2 = 48s+46$ ($m=12s+11$) and $s \not \equiv 2 \pmod{3}$.

Observe that $\dG =(24s+23)(48s+45) \equiv 3 \pmod{12}$.

\begin{lemma}
\label{I-gen}
Keep notations as above. If $n$ and $n+\dG$ are perfect numbers, then $n$ is odd and:
$$n = q^{4\bG +1}\cdot \prod p_u^{2\aG _u} = (2^{p-1}-2m-1)\cdot (2^p+4m+1)$$
where $Q = 2^p-1$ is prime, $q \equiv 1 \pmod{4}$ and where $q$ and the $p_u$'s are distinct odd primes.

Moreover if $l=8m+3$ is prime then:
\begin{equation}
\label{eq:p-1 I gen}
2^{p-1}-2m-1 = \prod p_i^{2a_i} =: u^2
\end{equation}
\begin{equation}
\label{eq: p I gen}
2^p+4m+1 = q^{4\bG +1}\cdot \prod p_j^{2a_j} =: qv^2
\end{equation}
and $q \equiv 5 \pmod{8}$.
\end{lemma}

\begin{proof} First let us show that $n$ is odd and $n+\dG$ is even. Since $\dG \equiv 3 \pmod{4}$, this is clear if $n > 6$. If $n=6$, $6+\dG$ is odd. Since $\dG \equiv 0 \pmod{3}$, $3^2 \mid 6+\dG$, but this is impossible since $s \not \equiv 2\pmod{3}$.
 
So $n +\dG = Q(Q+1)/2$ where $Q = 2^p-1$ is prime and $n=q^{4\bG +1}\cdot \prod p_u^{2\aG _u}$. We have:
$$\frac{Q(Q+1)}{2}- \frac{b(b-1)}{2}=\frac{(Q+b)(Q-b+1)}{2}$$
It follows that $n = A\cdot B$ with $A= (2^{p-1}-2m-1)$, $B= (2^p+4m+1)$. Since $B-2A = 8m+3 =l$, the g.c.d., $G$, of $A, B$ is $1$ or $l$ since $l$ is prime by assumption.

If $G=l$, since $l\equiv 3 \pmod{4}$, $l \neq q$. If $q \nmid B$, then $B = 2^p+4m+1 = l^c\cdot \prod p_j^{2a_j}$ with $c \in\{1,2a-1\}$ ($l^{2a}\parallel n$). Since, by assumption, $m \equiv 3 \pmod{4}$, $2^p+4m+1 \equiv 5 \pmod{8}$. Since $l\equiv 3 \pmod{8}$ and $c$ is odd, we get a contradiction. We conclude that $q \mid B$, hence $B =l^c\cdot q^{4\bG +1}\cdot \prod p_j^{2a_j}$ and we get: $5 \equiv B \equiv 3\cdot q^{4\bG +1}\equiv 3\cdot q \pmod{8}$. Since $q\equiv 1 \pmod{4}$, this is impossible. This shows that $G=1$. 

So one of $A,B$ is a square and the other is divisible by $q^{4\bG +1}$. Assume $B = 2^p+4m+1 = \prod p_j^{2a_j}$, then $5 \equiv 2^p+4m+1 \equiv \prod p_j^{2a_j} \equiv 1 \pmod{8}$, so $A$ is a square, $q \mid B$ and $q \equiv 5 \pmod{8}$.
\end{proof}

Now we have:

\begin{theorem}
\label{Thm-infty}
There exist infinitely many triangular numbers $\dG$, with $\dG \equiv 3\pmod{12}$, such that the equation $x - y=\dG$ has no solution with $x$ and $y$ both perfect numbers.
\end{theorem}

\begin{proof} We keep notations and assumptions of Lemma \ref{I-gen}. We have $l=8m+3 = 96s+91$. If $s=3a$, $l=288a+91$; if $s=3a+1$, $l=288a+187$. Since $(288,91)=1$ and $(288,187)=1$, by Dirichlet's theorem there are infinitely many values of $s$ for which $l$ is prime.

If $l$ is prime, by Lemma \ref{I-gen}: $2^{p-1}-2m-1 = x^2$ and $2^p+4m+1 = qy^2$. Adding these two equations we get: 
$$3\cdot 2^{p-1}+2m = x^2+qy^2$$
it follows, since $m\equiv 2 \pmod{3}$, that:
$$1 \equiv x^2+qy^2 \pmod{3}$$
\begin{itemize}
\item If $x^2\equiv 0 \pmod{3}$: then $2^{p-1}-2m-1 \equiv 0 \pmod{3}$ and it follows that $2^{p-1}\equiv 2 \pmod{3}$, which is impossible because $p-1$ is even.
\item If $y^2 \equiv 0 \pmod{3}$: then $2^p+4m+1 \equiv 0 \pmod{3}$, which implies $2^p\equiv 0 \pmod{3}$: impossible.
\item We conclude that $x^2\equiv y^2 \equiv 1 \pmod{3}$, but then we get $q\equiv 0 \pmod{3}$, i.e. since $q$ is prime $q=3$, which is also impossible because $q \equiv 1 \pmod{4}$
\end{itemize}
We conclude that if $\dG = (24s+23)\cdot (48s+45)$ and if $l=96s+91$ is prime, then $|x-y|=\dG$ has no solutions with both $x$ and $y$ perfect numbers.
\end{proof}

\begin{remark}
Of course we can get other examples. If $b=30, m=7$, then $8m+3=59$ is prime. Since $\dG =435$, $n$ is odd and Lemma \ref{I-gen} applies. Hence it suffices to show that $2^{p-1}-15$ is not a square. Now if $2^{p-1}-15 =x^2$, we have $(2^{\frac{p-1}{2}} -x)(2^{\frac{p-1}{2}} +x) = 15 = 3\cdot 5$ and it follows that $p=5$ or $p=7$. Since these values of $p$ are easily excluded in our context, we conclude that $\dG = 435$ can't be the distance between two perfect numbers.
\vspace{0.3cm}

The question of whether an odd triangular number can be the distance between two perfect numbers remains open.
\end{remark}

\end{document}